\documentclass[a4paper,12pt]{article}
\usepackage{amsmath}
\usepackage{amssymb}
\usepackage{amsfonts,times}
\usepackage{theorem}
\usepackage{hyperref}
\AtBeginDocument{}
\newtheorem{theorem}{Theorem}

\newtheorem{corollary}[theorem]{Corollary}

\newtheorem{definition}[theorem]{Definition}

\newtheorem{lemma}[theorem]{Lemma}

\newtheorem{proposition}[theorem]{Proposition}
{\theorembodyfont{\upshape}\newtheorem{remark}[theorem]{Remark}}
{\theorembodyfont{\upshape}}

\newenvironment{proof}[1][Proof]{\noindent\textbf{#1.} }{\ \hfill \rule{0.5em}{0.5em}}

\newcommand{\cF}{\mathcal{F}}

\newcommand{\cP}{\mathcal{P}}

\newcommand{\IN}{\mathbb{N}}
\newcommand{\IE}{\mathbb{E}}
\newcommand{\IP}{\mathbb{P}}
\newcommand{\IR}{\mathbb{R}}

\newcommand{\oo}[1]{\frac{1}{#1}}
\begin{document}

\title{An elementary proof of de Finetti's Theorem}
\author{Werner Kirsch
\\
Fakult\"{a}t f\"{u}r Mathematik und Informatik\\
FernUniversit\"{a}t in Hagen, Germany}

\maketitle
\begin{abstract}
 A sequence of random variables is called exchangeable if the joint distribution of the sequence is unchanged by any permutation of the indices.
De Finetti's theorem characterizes all $\{0,1\}$-valued exchangeable sequences as a `mixture' of sequences of independent random variables.
 
 We present an new, elementary proof of de Finetti's Theorem. The purpose of this paper is to make this theorem accessible to a broader community through an essentially self-contained proof.
\end{abstract}

\section{Introduction}
\begin{definition}
   A finite sequence of (real valued) random variables $X_{1},X_{2},\ldots,X_{N}$ on a probability space $(\Omega,\cF,\IP) $ is called \emph{exchangeable}, if for any permutation $\pi$ of $\{1,2,\ldots,N\}$ the
   distributions of $X_{\pi(1)},X_{\pi(2)},\ldots,X_{\pi(N)}$ and $X_{1},X_{2},\ldots,X_{N}$ agree, i.\,e. if for any Borel sets $A_{1},A_{2},\ldots,A_{N}$
   \begin{align}
      &\IP\big(X_{1}\in A_{1},X_{2}\in A_{2},\ldots,X_{N}\in A_{N}\big)\notag\\~=~&\IP(X_{\pi(1)}\in A_{1},X_{\pi(2)}\in A_{2},\ldots,X_{\pi(N)}\in A_{N})
   \end{align}

   An infinite sequence $\{X_{i}\}_{i\in\IN}$ is called exchangeable, if the finite sequences  $X_{1},X_{2},\ldots,X_{N}$ are exchangeable for any $N\in\IN$.
\end{definition}
Obviously, independent, identically distributed random variables are exchangeable, but there are many more examples of exchangeable sequences.

Let us denote by $\pi_{p}$ the (Bernoulli) probability measure on $\{0,1\}$ given by $\pi_{p}(1)=p$ and $\pi_{p}(0)=1-p$.
If the random variables $X_{i}$ are independent and distributed according to $\pi_{p}$, i.\,e. $\IP(X_{i}=1)=\pi_{p}(1)=p$
and $\IP(X_{i}=0)=\pi_{p}(0)=1-p$, then the probability distribution of the sequence $X_{1},\ldots,X_{N}$ is the product measure
\begin{equation}
   \cP_{p}~=~\bigotimes_{i=1}^{N}\;\pi_p\qquad\text{on}\quad\{0,1\}^{N}
\end{equation}

In 1931 B. de Finetti proved the following remarkable theorem which now bears his name:

\begin{theorem}[de Finetti's Representation Theorem]\label{deFin}
Let $X_{i}$ be an infinite sequence of $\{0,1\}$-valued exchangeable random variables then there exists a probability measure $\mu$ on $[0,1]$ such that
for any $N$ and any sequence $(x_{1},\ldots,x_{N})\in\{0,1\}^{N}$
\begin{align}
   \IP\big(X_{1}=x_{1},\ldots,X_{N}=x_{N}\big)~&=~\int\,\cP_{p}(x_{1},\ldots,x_{N})\,d\mu(p)\\
   &=~\int\,\prod_{i=1}^{N}\pi_{p}(x_{i})\,d\mu(p)
\end{align}

\end{theorem}

Loosely speaking: An exchangeable sequence with values in $\{0,1\}$ is a `mixture' of independent sequences with respect to a measure $\mu$ on $[0,1]$. 

De Finetti's Theorem was extended in various directions, most notably to random variables with values in rather general spaces \cite{Hewitt}. For reviews
on the theorem see e.\,g. \cite{Aldous}, see also the textbook \cite{Klenke} for a proof.

The proof of Theorem \ref{deFin} we present here is very elementary. It is based on the method of moments which allows us to prove weak convergence of measures.

\medskip
\noindent\textbf{Acknowledgement} It is a pleasure to thank Michael Fleermann for careful proofreading and many helpful suggestions.

\section{Preliminaries}

For a probability measure $\mu$ on $\IR$ we define the $k^{th}$ moments by $m_{k}(\mu):=\int x^{k}\,d\mu(x)$ whenever the latter integral exists
(in the sense that $\int |x|^{k}\,d\mu(x)<\infty$).
In the following we will be dealing with measures with compact support so that all moments exist (and are finite). The following theorem is a
light version of the method of moments which is nevertheless sufficient for our purpose.

\begin{proposition}\label{bdsupp}\
\begin{enumerate}
\item\label{a} Let $\mu_{n}$ ($n\in\IN$) be probability measures with support contained in a (fixed) interval $[a,b]$.
If for all $k$ the moments $m_{k}(\mu_{n})$ converge to some $m_{k}$ then the sequence $\mu_{n}$ converges weakly
to a measure $\mu$ with moments $m_{k}(\mu)=m_{k}$ and with support contained in $[a,b]$.
\item\label{b} If $\mu$ is a probability measure with support contained in $[a,b]$ and $\nu $ is a probability measures on $\IR$ such that
$m_{k}(\mu)=m_{k}(\nu)$ then $\mu=\nu$.
\end{enumerate}

\end{proposition}

\begin{remark}
    Let $\mu_{n}$ and $\mu $ be probability measures on $\IR$.
    Recall that weak convergence of the measures $\mu_{n}$ to $\mu $ means that
    \begin{align}
       \int f(x)\,d\mu_{n}(x)~\Rightarrow~\int f(x)\,d\mu(x)
    \end{align}
    for all bounded, continuous functions $f$ on $\IR$.
    
   The above theorem is true and, in fact, well known if the support condition is replaced by the much weaker assumption that the moments $m_{k}(\mu)$ (resp. the numbers $m_{k}$) do not grow too fast as $k\to\infty$
   (see \cite{Klenke} or \cite{Kirsch} for details).
\end{remark}

\begin{proof}
   We sketch the proof, for details see the literature cited above.

   \noindent By Weierstrass approximation theorem the polynomials on $I=[a-1,b+1]$ are uniformly dense in the space of continuous functions on $I$. Hence the integral
   $\int f(x) d\mu(x)$ for continuous $f$ can be computed from the knowledge of the moments of $\mu$. From this part \ref{b} of the theorem follows.

   Moreover, we get that the integrals $\int f(x)\,d\mu_{n}(x)$ converge for any continuous $f$. The limit is a positive linear functional. Thus the probability measures $\mu_{n}$ converge weakly to a
   measure $\mu $ with
   \begin{equation*}
      \int f(x)\,d\mu_{n}(x)~\to~\int f(x)\,d\mu(x)
   \end{equation*}
   which implies part \ref{a}.
\end{proof}
\section{Proof of de Finetti's Theorem}

The following theorem is a substitute for a (very weak) law of large numbers.
\begin{theorem}\label{Fmeas}
  Let $X_{i}$ be an infinite sequence of $\{0,1\}$-valued exchangeable random variables then
  $S_{N}:=\frac{1}{N}\sum_{i=1}^{N}\,X_{i}$ converges in distribution to a probability measure $\mu $.\\
  $\mu$ is concentrated on $[0,1]$ and its moments are given by
  \begin{equation}
     m_{k}(\mu)~=~\IE\Big(X_{1}\cdot X_{2}\cdot\ldots\cdot X_{k}\Big)
  \end{equation}
  where $\IE$ denotes expectation with respect to $\IP$.
\end{theorem}
\begin{definition}\label{defFmeas}
   We call the measure $\mu$ associated with $X_{i}$ according to Theorem \ref{Fmeas} the \emph{de Finetti measure} of $X_{i}$.
\end{definition}
\begin{proof}{(Theorem \ref{Fmeas})}

   To express the moments of $S_{N} $ we compute
   \begin{align}\label{binomi}
      \Big(\sum_{i=1}^{N}X_{i}\Big)^{k}~&=~\sum_{(i_{1},\ldots,i_{k})\in\{1,\ldots,N\}^{k}} X_{i_{1}}\cdot X_{i_{2}}\cdot\ldots\cdot X_{i_{k}}
   \end{align}
   To simplify the evaluation of the above sum we introduce the number of different indices in $(i_{1},\ldots,i_{k})$ as
   \begin{align}\label{eq:rho}
      \rho(i_{1},i_{2},\ldots,i_{k})~=~\#\{i_{1},i_{2},\ldots,i_{k}\}
   \end{align}
  Consequently

   \begin{align}
      \eqref{binomi}&=~\sum_{r=1}^{k}\,\sum_{\overset{(i_{1},\ldots,i_{k})\in\{1,\ldots,N\}^{k}}{\rho(i_{1},\ldots.i_{k})=r}} X_{i_{1}}\cdot X_{i_{2}}\cdot\ldots\cdot X_{i_{k}}
   \end{align}
   Thus we may write
   \begin{align}
      \IE\Big(\big(\oo{N}\sum_{i=1}^{N}X_{i}\big)^{k}\Big)~=~~&\oo{N^{k}}\,\sum_{\overset{i_{1},\ldots,i_{k}=1}{\rho(i_{1},\ldots,i_{k})=k}}^{N} \IE\Big(X_{i_{1}}\cdot X_{i_{2}}\cdot\ldots\cdot X_{i_{k}}\Big)\notag\\
      +\,&\oo{N^{k}}\,\sum_{\overset{i_{1},\ldots,i_{k}=1}{\rho(i_{1},\ldots,i_{k})<k}}^{N} \IE\Big(X_{i_{1}}\cdot X_{i_{2}}\cdot\ldots\cdot X_{i_{k}}\Big)\label{SN}
       \end{align}
       There are at most $(k-1)^{k}\,N^{k-1}$ index tuples $(i_{1},\ldots,i_{k})$ with $\rho(i_{1},\ldots,i_{k})<k$. Indeed, we have $N^{k-1}$ possibilities  to chose the possible indices (`candidates') for $(i_{1},\ldots,i_{k})$. Then for each of the $k$ positions in the $k$-tuple we may chose one of the $k-1$ candidates which gives $(k-1)^{k}$ possibilities. This covers also tuples with less than $k-1$ different indices as
       some of the candidates may finally not appear in the tuple.  It follows that the second term in \eqref{SN} goes
   to zero. So
      \begin{align}
      \IE\Big(\big(\oo{N}\sum_{i=1}^{N}\big)^{k}\Big)~&\approx~\oo{N^{k}}\sum_{\overset{i_{1},\ldots,i_{k}=1}{\rho(i_{1},\ldots,i_{k})=k}}^{N} \IE\Big(X_{i_{1}}\cdot X_{i_{2}}\cdot\ldots\cdot X_{i_{k}}\Big)\,,\notag\\
      \intertext{so using exchangeability:}
      &=~\oo{N^{k}}\sum_{\overset{i_{1},\ldots,i_{k}=1}{\rho(i_{1},\ldots,i_{k})=k}}^{N}\,\IE\Big(X_{1}\cdot X_{2}\cdot\ldots\cdot X_{k}\Big)\notag\\
      &\approx~\IE\Big(X_{1}\cdot X_{2}\cdot\ldots\cdot X_{k}\Big)
   \end{align}
     An application of Proposition \ref{bdsupp} gives the desired result.
\end{proof}

We note a Corollary to the  Theorem \ref{Fmeas} or better to its proof.
\begin{corollary}\label{cor}
   If $\{X_{i}\} $ is an exchangeable sequence of $\{0,1\} $-valued random variables and 
   \begin{align}\notag
   r=\rho(i_{1},i_{2},\ldots,i_{k})~=~\#\{i_{1},i_{2},\ldots,i_{k}\}
   \end{align}
   then
   \begin{align}
      \IE\Big(X_{i_{1}}\cdot X_{i_{2}}\cdot\ldots\cdot X_{i_{k}}\Big)~=~\IE\Big(X_{1}\cdot X_{2}\cdot\ldots\cdot X_{r}\Big)
   \end{align}
\end{corollary}

\begin{proof}
   Since $X_{i}\in\{0,1\}$ we have ${X_{i}}^{\ell}~=X_{i}$ for all $\ell\in\IN, \ell\geq 1$, Hence, the product in the left hand side is actually a product of $r$ different
   $X_{j}$, the expectation of which equals the right hand side due to exchangeability.
\end{proof}

For the proof of Theorem \ref{deFin} we will use the following simple lemma.
\begin{lemma}\label{sum}
   Suppose $\{X_{i}\}_{i\in\IN}$ is a $\{0,1\}$-valued exchangeable sequence. Then for pairwise distinct $i_{1},i_{2},\ldots,i_{k}\in\IN$ and $x_{1},\ldots,x_{k}\in\{0,1\}$
   with $\sum x_{i}=m$
   \begin{align}
      \IP\Big(X_{i_{1}}=x_{1},\ldots,X_{i_{k}}=x_{k}\Big)~=~\oo{\binom{k}{m}}\;\IP\Big(\sum_{i=1}^{k}X_{i}=m\Big)
   \end{align}
\end{lemma}
\begin{proof}
   There are $\binom{k}{m}$ tuples $x_{1},\ldots,x_{k}$ with $\sum x_{i}=m$. Due to exchangeability they all lead to the same probability.
\end{proof}

We now prove Theorem \ref{deFin}.

\begin{proof}{(Theorem \ref{deFin})}

Let $\mu $ be the de Finetti measure of $X_{i}$ (see Definition \ref{defFmeas}) and define a $\{0,1\}$-valued process $\{Y_{i}\}_{i}$ by
\begin{align}\label{defY}
   \IP\Big( Y_{1}=y_{1},\ldots, Y_{k}=y_{k}\Big)~=~\int \prod_{i=1}^{k}\pi_{p}(y_{i})\,d\mu(p)
\end{align}
The process $Y_{i}$ is obviously exchangeable.               

We'll prove that $X_{i}$ and $Y_{i}$ have the same finite dimensional distributions.

According to Lemma \ref{sum} it suffices to show that $S_{N}=\sum_{i=1}^{N}X_{i}$ and $T_{N}=\sum_{i=1}^{N}Y_{i}$ have the same distributions for all $N$ and for
this it is enough by Proposition \ref{bdsupp} to prove that their moments agree.
\begin{align}
   \IE\Big({S_{N}}^{k}\Big)~&=~\sum_{r=1}^{k}\,\sum_{\overset{i_{1},\ldots,i_{k}=1}{\rho(i_{1},\ldots,i_{k})=r}}^{N}\,\IE\big(X_{i_{1}}\cdot\ldots\cdot X_{i_{k}}\big)\notag\\
   &=~\sum_{r=1}^{k}\,\sum_{\overset{i_{1},\ldots,i_{k}=1}{\rho(i_{1},\ldots,i_{k})=r}}^{N}\,\IE\big(X_{1}\cdot\ldots\cdot X_{r}\big) \tag{\text{by Corollary \ref{cor}}}\\
   &=~\sum_{r=1}^{k}\,\sum_{\overset{i_{1},\ldots,i_{k}=1}{\rho(i_{1},\ldots,i_{k})=r}}^{N}\,\int\,\prod_{i=1}^{r}\pi_{p}(x_{i})\,d\mu(p)
   \tag{\text{by Theorem \ref{Fmeas}}}\\
   &=~\sum_{r=1}^{k}\,\sum_{\overset{i_{1},\ldots,i_{k}=1}{\rho(i_{1},\ldots,i_{k})=r}}^{N}\,\IE\big(Y_{1}\cdot\ldots\cdot Y_{r}\big)\tag{by \eqref{defY}}\\
   &=~\sum_{r=1}^{k}\,\sum_{\overset{i_{1},\ldots,i_{k}=1}{\rho(i_{1},\ldots,i_{k})=r}}^{N}\,\IE\big(Y_{i_{1}}\cdot\ldots\cdot Y_{i_{k}}\big)\tag{Corollary \ref{cor}}\\
   ~&=~ \IE\Big({T_{N}}^{k}\Big)
\end{align}
\end{proof}

\bigskip\bigskip

\noindent
\begin{tabular}{lcl}
\textbf{Werner Kirsch}&\quad& \texttt{werner.kirsch@fernuni-hagen.de}
\end{tabular}

\begin{thebibliography}{99}

\bibitem{Aldous} D. Aldous: \textit{Exchangeability and related topics},
pp.~1-198 in: Lecture Notes in Mathematics 117, Springer (1985).


\bibitem{deFinetti} B. de Finetti: Funzione caratteristica di un fenomeno aleatorio, Atti della R.
Accademia Nazionale dei Lincei, Ser. 6, Memorie, Classe di Scienze Fisiche, Matematiche e
Naturali 4, 251--299 (1931).

\bibitem{deFinetti2} B. de Finetti: La prevision: ses lois logiques, ses sources subjectives, Annales de l'lnstitut Henri
Poincare, 7, 1--68  (1937).

\bibitem{Hewitt} Hewitt, E. and Savage, L. J.: Symmetric measures on Cartesian products. Transactions of the American Mathematical Society, 80, 470--501 (1955) .

\bibitem{Kirsch} W. Kirsch: \textit{Moments in Probability}, book in
preparation,\\ to appear at DeGruyter.

\bibitem{Klenke} A. Klenke: Probability Theory, Springer (2013)




\end{thebibliography}
\end{document}